\documentclass[12pt]{amsart}
\usepackage{enumerate,amstext,amsfonts,amssymb,amsbsy,amsmath,verbatim}
\usepackage{ifthen,multicol, extarrows}
\usepackage[all]{xy}
\usepackage[bookmarks, colorlinks=true, linkcolor=blue, citecolor=blue, urlcolor=blue]{hyperref}

\usepackage{eucal}
\usepackage[margin=1in]{geometry}

\numberwithin{equation}{section}
\newtheorem{lemma}[equation]{Lemma}

\newtheorem{prop}[equation]{Proposition}
\newtheorem{cor}[equation]{Corollary}

\newtheorem{claim*}{Claim}
\newtheorem{thm}[equation]{Theorem}

\theoremstyle{definition}
\newtheorem{rmk}[equation]{Remark}
\newenvironment{remark}[1][]{\begin{rmk}[#1] \pushQED{\qed}}{\popQED \end{rmk}}
\newtheorem{eg}[equation]{Example}
\newenvironment{example}[1][]{\begin{eg}[#1] \pushQED{\qed}}{\popQED \end{eg}}
\newtheorem{defn-placeholder}[equation]{Definition}
\newenvironment{defn}[1][]{\begin{defn-placeholder}[#1]\pushQED{\qed}}{\popQED \end{defn-placeholder}}

\newcommand{\Spec}{\operatorname{Spec}}

\newcommand{\coker}{\operatorname{coker}}

\newcommand{\Hom}{\operatorname{Hom}} 
\newcommand{\Ext}{\operatorname{Ext}} 

\newcommand{\rank}{\operatorname{rank}}

\newcommand{\Sym}{\operatorname{Sym}} 
\newcommand{\GL}{\mathbf{GL}}

\newcommand{\defi}[1]{{\bf \textsf{#1}}} 

\newcommand{\DD}{\mathrm D}

\newcommand{\PP}{\mathbb P}
\renewcommand{\AA}{\mathbb A}
\newcommand{\QQ}{\mathbb Q}
\newcommand{\ZZ}{\mathbb Z}

\newcommand{\bd}{\mathbf d}
\newcommand{\be}{\mathbf e}
\newcommand{\bE}{\mathbf E}
\newcommand{\bF}{\mathbf F}
\newcommand{\bk}{\mathbf{k}}
\newcommand{\bS}{\mathbf{S}}

\newcommand{\rE}{\mathrm{E}}
\newcommand{\rH}{\mathrm{H}}
\newcommand{\rh}{\mathrm{h}}
\newcommand{\rK}{\mathrm{K}}
\newcommand{\rR}{\mathrm R}

\newcommand{\cA}{\mathcal{A}}
\newcommand{\cE}{\mathcal E}
\newcommand{\cF}{\mathcal F}
\newcommand{\cG}{\mathcal G}
\newcommand{\cO}{\mathcal O}
\newcommand{\cQ}{\mathcal{Q}}
\newcommand{\cU}{\mathcal U}
\newcommand{\cW}{\mathcal W}

\newcommand{\fp}{\mathfrak{p}}
\newcommand{\fgl}{\mathfrak{gl}}

\def\reg{\operatorname{reg}}

\newcommand{\arxiv}[1]{\href{http://arxiv.org/abs/#1}{{\tt arXiv:#1}}}

\title{Supernatural analogues of Beilinson monads}
\author{Daniel Erman}
\address{Department of Mathematics, University of Wisconsin, Madison, WI}
\email{\href{mailto:derman@math.wisc.edu}{derman@math.wisc.edu}}
\urladdr{\url{http://math.wisc.edu/~derman/}}

\author{Steven V Sam}
\address{Department of Mathematics, University of California, Berkeley, CA}
\curraddr{Department of Mathematics, University of Wisconsin, Madison, WI}
\email{\href{mailto:svs@math.wisc.edu}{svs@math.wisc.edu}}
\urladdr{\url{http://math.wisc.edu/~svs/}}

\thanks{DE was partially supported by NSF grant DMS-1302057. SS was supported by a Miller research fellowship.}

\date{October 17, 2016}



\begin{document}

\maketitle

\begin{abstract}
We use supernatural bundles to build $\GL$-equivariant resolutions supported on the diagonal of $\PP^n\times \PP^n$, in a way that extends Beilinson's resolution of the diagonal.  We thus obtain results about supernatural bundles that largely parallel known results about exceptional collections.  We apply this construction to Boij--S\"oderberg decompositions of cohomology tables of vector bundles, yielding a proof of concept for the idea that those positive rational decompositions should admit meaningful categorifications.  
\end{abstract}


\section{Introduction}

Throughout we work over a field $\bk$ of characteristic zero.  In this paper, we use supernatural bundles to produce $\GL$-equivariant resolutions supported along the diagonal of $\PP^n\times \PP^n$, in a way that extends Beilinson's resolution of the diagonal.  We thus obtain results about supernatural bundles that largely parallel known results about exceptional collections, including analogues of Beilinson monads.

We then apply these resolutions to the study of Boij--S\"oderberg decompositions of vector bundles.  
Boij--S\"oderberg theory originated in~\cite{boij-sod-1} with a complete conjectural classification (up to scalar multiple) of Betti tables of graded free resolutions of finite length modules over a polynomial ring.
The conjectures were proven by Eisenbud and Schreyer~\cite{ES2008}, who introduced a dual side to the theory that completely classified (up to scalar multiple) the cohomology tables of vector bundles on $\PP^n$.  In particular, \cite[Theorem~0.5]{ES2008} shows that the cohomology table of any vector bundle decomposes as a positive rational linear combination of the cohomology tables of supernatural vector bundles.  See~\cite{floystad, schreyer-eisenbud-icm} for an introduction to the theory, and Example~\ref{eg:BS-decomp} for an example of a Boij--S\"oderberg decomposition.

Perhaps the most mysterious question about Boij--S\"oderberg theory is whether these numerical decompositions, which very often involve rational coefficients, admit any sort of categorification to the level of vector bundles.  We provide an affirmative answer in many cases, showing that a Fourier--Mukai transform with respect to one of our equivariant resolutions naturally clears the denominators and categorifies many of these decompositions.

\subsection{Supernatural bundles as a parallel for exceptional bundles}
A vector bundle $\cF$ on $\PP^n$ is \defi{exceptional} if $\Ext^*(\cF,\cF)$ is as small as possible, i.e. if
\[
\dim \Ext^0(\cF,\cF)=1 \text{ and } \dim \Ext^i(\cF,\cF)=0 \text{ for } i>0.
\]
A vector bundle $\cF$ on $\PP^n$ is \defi{supernatural} if it has as little cohomology as possible, i.e. if for each $j\in \ZZ$ there is at most one $i$ such that $\rH^i(\PP^n,\cF(j))\ne 0$, and if the Hilbert polynomial of $\cF$ has $n$ distinct integral roots.  

Supernatural bundles were first defined in~\cite[p.~862]{ES2008} (a closely related definition appeared in~\cite[p.365]{hartshorne-hirschowitz}) where they played a key role in the main results of Boij--S\"oderberg theory, providing the extremal rays of the cone of cohomology tables. In characteristic $0$, many supernatural bundles are familiar objects: consider the tautological exact sequence of vector bundles on $\PP^n$
\[
0\to \cO(-1)\to \cO^{n+1} \to \mathcal Q \to 0;
\]
applying a Schur functor $\bS_{\lambda}$ to $\cQ^*$ gives a supernatural bundle.  Moreover, by varying $\lambda$ and by twisting by line bundles, one gets a supernatural bundle corresponding to each root sequence (see \S\ref{sec:background} for detailed definitions of root sequences).  However, we emphasize that this does not account for all known supernatural bundles.  By pushing forward line bundles from a product of projective spaces, one can construct supernatural bundles in any characteristic, and even in characteristic $0$ these bundles are generally distinct from the $\bS_{\lambda}\cQ^*$~\cite[\S6]{ES2008}; moreover, many supernatural bundles admit nontrivial moduli \cite[\S6]{ES-river}.

Although most supernatural bundles fail to be exceptional, there are some immediate commonalities with exceptional vector bundles.  To begin with, the basic examples of exceptional collections, namely $\{\cO(-i)\}_{i=0}^n$ and $\{\Omega^i(i)\}_{i=0}^n$, consist entirely of supernatural sheaves.
In addition, a chain of root sequences (see \S\ref{sec:background}) determines a sequence of supernatural bundles $\dots, \cE_i, \cE_{i+1}, \dots$ where $\Hom(\cE_i,\cE_j)\ne 0 \iff i\geq j$~\cite[Theorem~1.2]{beks-poset}. 

The following theorem parallels known results about exceptional sequences~{\cite{beilinson, BGG, GR-mutations}}.  We write $\rK_0(\PP^n)_{\QQ}:=\rK_0(\PP^n) \otimes_{\ZZ}\QQ$.

\begin{thm} \label{thm:main comparison}
Let $W = \{\cO(-w_0), \cO(-w_1), \dots, \cO(-w_n)\}$ be any collection of line bundles with $w_0<w_1<\dots < w_n$. Define $\mu(W)_i = w_n - w_{i-1} - (n-i+1)$ and 
\[
N_W := \dim \bS_{\mu(W)}(\bk^{n+1}) = \det \left( \binom{w_n - w_{i-1} + j - 1}{n}\right)_{i,j=1}^{n+1}.
\]
$($The inputs in the determinant are binomial coefficients; see \cite[Exercise A.30(iii)]{fulton-harris} for this dimension formula.$)$
\begin{enumerate}[\indent \rm (1)]
\item\label{QQbasis}  {\bf $\QQ$-basis}: $W$ is a basis for $\rK_0(\PP^n)_{\QQ}$; more precisely, $W$ spans a subgroup of index $N_W$ in $\rK_0(\PP^n)$.

\item\label{Orthogonal}  {\bf Orthogonal basis}: There is a second collection $W^\perp = \{\cE_0,\cE_1, \dots, \cE_n\}$, also spanning a subgroup of index $N_W$ in $\rK_0(\PP^n)$, which is orthogonal to $W$ in the following sense:
\[
\rH^i(\PP^n, \cE_k \otimes \cO(-w_j))= 
\begin{cases}
\bk^{N_W} & \text{if } i=j=k\\
0& \text{else}
\end{cases}.
\]
The objects of $W^\perp$ are $\GL$-equivariant supernatural bundles defined in \eqref{eqn:Wperp-defn}.

\item\label{Resolution}  {\bf Resolution of sheaf on diagonal:}  $W$ and $W^\perp$ can be combined on $\PP^n\times \PP^n$ to give a $\GL$-equivariant resolution
\[
\mathbf{E}_W:= [ \cO(-w_0)\boxtimes \cE_0\gets \cO(-w_1)\boxtimes \cE_1 \gets \cdots \gets  \cO(-w_n)\boxtimes \cE_n \gets 0]
\]
of a sheaf $\cU_W$ set-theoretically supported on $\Delta$ $($the diagonal copy of $\PP^n$ in $\PP^n\times \PP^n)$. The pushforward of $\cU_W$ to either copy of $\PP^n$ is a vector bundle of rank $N_W$.
\end{enumerate}
\end{thm}

On the one hand, this is a weaker result than the parallel result for exceptional sequences.  For instance, any full exceptional collection forms a $\ZZ$-basis of $\rK_0(\PP^n)$, and can be used to resolve $\cO_{\Delta}$ in the derived category (see, for example, \cite[\S 3]{canonaco}). Performing a Fourier--Mukai transform with respect to our resolution thus has a cost that is not present when working with Beilinson's resolution: whereas a Fourier--Mukai transform with respect to a resolution of $\cO_{\Delta}$ is the identity in the derived category, a Fourier--Mukai transform with respect to an appropriate resolution of a higher rank sheaf can introduce a ``scalar multiple''.

On the other hand, in contrast with the bundles that arise in exceptional sequences via mutations, we work entirely with simple, familiar bundles:  line bundles on the $W$ side, and equivariant bundles of the form $\bS_{\lambda} \mathcal Q^*$ on the $W^\perp$ side.

The simplest example of Theorem~\ref{thm:main comparison} is when $W=\{\cO(-i)\}_{i=0}^n$; then $W^\perp = \{\Omega^i(i)\}_{i=0}^n$ and $\bE_W$ is Beilinson's resolution of the diagonal~\cite{beilinson}.
The following is one of the next simplest examples.

\begin{example}\label{ex:023}
On $\PP^2$, let $W=\{\cO,\cO(-2),\cO(-3)\}$.  In this case $W^\perp =\{\cE_0,\cE_1,\cE_2\}$ where $\cE_0=\cO(1), \cE_1 = (\Sym^2 \cQ^*)(1)$ and $\cE_2 = \cQ^*=\Omega^1(1)$.  We then have:
\[
\bE_W = [\cO\boxtimes \cE_0 \gets \cO(-2)\boxtimes \cE_1\gets \cO(-3)\boxtimes \cE_2\gets 0],
\]
which resolves a sheaf $\cU_W$ set-theoretically supported on the diagonal copy of $\PP^2$ in $\PP^2\times \PP^2$.  In fact, if $\mathcal I_\Delta$ is the ideal sheaf of the diagonal, then $\cU_W=\cO_{\PP^2\times\PP^2}/\mathcal I_\Delta^2$, and the pushforward of $\cU_W$ to either factor is a rank $3$ bundle on $\PP^2$.
\end{example}

The following Fourier--Mukai transforms play an essential role in many of our applications.  Let $\DD^b(\PP^n)$ be the bounded derived category of coherent sheaves on $\PP^n$.  Following~\cite[Remark~5.2]{huybrechts}, we do not require a Fourier--Mukai transform to induce a derived equivalence.

\begin{defn}
For any $n$ and any $W$, we define the Fourier--Mukai transforms $\Phi^W_i \colon \DD^b(\PP^n)\to \DD^b(\PP^n)$ for $i=1,2$ via:
\[
\Phi^W_1(\cF) = \rR p_{1*}\left( p_2^*\cF \otimes \mathbf{E}_W\right) \qquad \text{ and } \qquad \Phi^W_2(\cF') = \rR p_{2*}\left( p_1^*\cF' \otimes \mathbf{E}_W\right). \qedhere
\]
\end{defn}

Recall that in \cite{beilinson}, Beilinson uses his resolution of the diagonal to construct two monads for an arbitrary sheaf; the first monad (sometimes referred to as {\em the Beilinson monad}) involves the sheaves $\{\Omega^i(i)\}_{i=0}^n$ and the second monad involves the line bundles $\{\cO(-i)\}_{i=0}^n$; see also~\cite{ancona-ottaviani, EFS} for additional details.   Our first application of Theorem~\ref{thm:main comparison} yields two spectral sequences which provide analogues of these monads for an arbitrary $W$.  We write $[\cF]$ for the class of a sheaf $\cF$ in the Grothendieck group $\rK_0(\PP^n)$.

\begin{cor}\label{cor:categorification}
Keep the notation of Theorem~\ref{thm:main comparison}.  Let $\cF$ be a coherent sheaf on $\PP^n$.
\begin{enumerate}[\indent \rm (1)]
\item   Let $i=1$ or $2$. Then $\Phi^W_i \colon \rK_0(\PP^n)\to \rK_0(\PP^n)$ is multiplication by $N_W$; in particular
\[
[\Phi^W_i(\cF)] = N_W \cdot [\cF].
\]

\item  The expressions for $[\cF]$ in terms of the bases $W$, $W^\perp$ for $\rK_0(\PP^n)_{\QQ}$ are given by:
  \begin{enumerate}[\rm (a)]
  \item $[\cF] = \frac{1}{N_W} \sum_{j=0}^n (-1)^j \chi(\PP^n,\cF\otimes \cE_j) \cdot [\cO(-w_j)],$ and
  \item  $[\cF] = \frac{1}{N_W}\sum_{j=0}^n (-1)^j \chi(\PP^n,\cF(-w_j)) \cdot [\cE_j].$
\end{enumerate}

\item There are spectral sequences for computing $\Phi^W_i(\cF)$ that categorify each of these expressions.  Namely, there are spectral sequences $\rE^1_{p,q}$ and $\widehat{\rE}^1_{p,q}$ where:
\begin{enumerate}[\rm (a)]
\item $\rE^1_{p,q}= \cO(-w_q)\otimes \rH^{-p}(\PP^n, \cE_q\otimes \cF) \Longrightarrow \Phi^W_1(\cF)$, and 
\item $\widehat{\rE}^1_{p,q} = \rH^{-p}(\PP^n, \cF(-w_q)) \otimes \cE_q \Longrightarrow \Phi^W_2(\cF)$.
\end{enumerate}
\end{enumerate}
\end{cor}

We also obtain an analogue of the existence of linear resolutions.  A $0$-regular sheaf $\cF$ has a linear resolution,
i.e. a resolution of the form
\[
\cO^{b_0}\gets \cO(-1)^{b_1}\gets \dots \gets \cO(-n)^{b_n}\gets 0;
\]
see for instance ~\cite[Proposition~1.8.8]{lazarsfeld}.  This can also be shown using Beilinson's monad in terms of $\{\cO(-i)\}_{i=0}^n$.  

We generalize this result to arbitrary collections of line bundles as follows.  We will say that a sheaf $\cF$ has a \defi{pure resolution}\footnote{Pure resolutions play a central role in Boij--S\"oderberg theory, where they are, in a certain sense, dual to supernatural vector bundles~\cite{ES2008,eisenbud-erman-categorified}.} of type $(w_0,w_1,\dots,w_n)$ if there exists a resolution of the form:
\[
\cO(-w_0)^{b_0}\gets \cO(-w_1)^{b_1}\gets \dots \gets \cO(-w_n)^{b_n}\gets 0.
\]
For an arbitrary collection $W=\{\cO(-w_0),\cO(-w_1),\dots,\cO(-w_n)\}$, it is not hard to find examples where $\cF$ cannot have a resolution in terms of these line bundles.  In particular, there can be no such resolution whenever the expression for the class $[\cF]$ in terms of the $[\cO(-w_i)]$ in $\rK_0(\PP^n)_{\QQ}$ involves non-integral coefficients.

However, Corollary~\ref{cor:categorification} shows that the Fourier--Mukai transform $\Phi^W_1$ simultaneously clears all of the denominators for the expression of $[\cF]$. Once this obstruction is removed by $\Phi^W_1$, we obtain a pure resolution of type $(w_0,w_1,\dots,w_n)$ of any $w_0$-regular sheaf $\cF$.

\begin{cor} \label{cor:pure resolutions}
Let $\cF$ be any $w_0$-regular coherent sheaf on $\PP^n$ and fix any sequence $w_0 < w_1 < \dots < w_n$.  Let $W=\{\cO(-w_0),\cO(-w_1),\dots,\cO(-w_n)\}$.  Then $\Phi^W_1(\cF)$ admits a pure resolution of type $(w_0,w_1,\dots,w_n)$.
\end{cor}
\noindent See also Example~\ref{ex:023 line}.
\subsection{Categorified Boij--S\"oderberg decompositions}\label{subsec:intro categorified}

The cohomology table $\gamma(\cF)$ of a vector bundle $\cF$ on $\PP^n$ is a table whose entries record the dimensions of all of the cohomology groups of $\cF$, with respect to twists of $\cF$ by all line bundles $\cO(j)$.  The $(i,j)$-entry of $\gamma(\cF)$ is given by the formula
\[
\gamma_{i,j}(\cF) := \rh^i(\PP^n,\cF(j)).
\]
Eisenbud and Schreyer show that the cohomology table of any vector bundle decomposes as a positive rational linear combination of the cohomology tables of supernatural vector bundles~\cite[Theorem~0.5]{ES2008}.

\begin{example} \label{eg:BS-decomp}
\addtocounter{equation}{-1}
\begin{subequations}
If $\cF$ is the cokernel of a generic matrix $\cO_{\PP^2}(-1)^5 \gets \cO_{\PP^2}(-2)^2$, then $\cF$ is a rank $3$ vector bundle on $\PP^2$. Using the short exact sequence
\[
0\gets \cF \gets \cO_{\PP^2}(-1)^5 \gets \cO_{\PP^2}(-2)^2 \gets 0
\]
and the genericity of the matrix, one can compute the cohomology groups $\rH^i(\PP^2, \cF(j))$ for all $i$ and $j$.  For instance, if we twist by $\cO(1)$ we see immediately that $\dim \rH^i(\cF(1))$ is $5$ if $i=0$ and is zero for $i=1,2$.

Boij--S\"oderberg theory enables us to succinctly encode all of the cohomology groups of $\cF$ via a convex, rational sum of supernatural bundles.  In this example, a direct computation shows that the supernatural bundles $\cE:=\cQ^*$ and $\cE':=(\Sym^2 \cQ^*)(1)$ will appear in the Boij--S\"oderberg decomposition of $\cF$, and that the decomposition works out to be:
\begin{equation}\label{eqn:gamma add}
\gamma(\cF) = \gamma(\cE) + \tfrac{1}{3}\gamma(\cE').
\end{equation}
The appearance of rational coefficients is very common in those sorts of computations.  
\end{subequations}
\end{example}

A fundamental mystery raised by Boij--S\"oderberg theory is whether these numerical decompositions admit meaningful categorifications to the level of vector bundles.

There is an obvious obstacle to categorifying \eqref{eqn:gamma add}: the non-integral coefficient.  
Naively, we might replace $\cF$ by $\cF^{\oplus 3}$ and hope for a splitting $\cF^{\oplus 3}\cong \cE^{\oplus 3}\oplus \cE'$, but that line of thinking has not produced meaningful categorifications.  
There has been work on categorification on the dual side of Boij--S\"oderberg theory involving Betti tables~\cite{ees-filtering}, but those results never involve cases with a non-integral coefficient.

There is also a more subtle obstacle to categorifying Boij--S\"oderberg decompositions which is given by the fact that supernatural bundles themselves can have moduli.  

The Fourier--Mukai transform $\Phi_2^W$ can address both of these obstacles.  The sheaf $\cU_W$ is a twisted Ulrich sheaf (see Proposition~\ref{prop:ulrich}), and thus---as already observed in Corollary~\ref{cor:categorification}---the transform $\Phi^W_2$ has an effect similar to scalar multiplication.  Moreover, we will see in Corollary~\ref{cor:supernatural becomes equivariant} that $\Phi^W_2$ also addresses the second obstacle by entirely collapsing the moduli of certain supernatural bundles.

\begin{example} \label{eg:BS-cat}
Returning to Example~\ref{eg:BS-decomp}, the transformed bundle $\Phi_2^W(\cF)$ splits as:
\[
\Phi_2^W(\cF)\cong \cE^{\oplus 3}\oplus  \cE'.
\]
This implies \eqref{eqn:gamma add}, yielding the desired categorification.  See Example~\ref{eg:BS-explain} for details.
\end{example}

This example is a special case of the following result.  

\begin{thm} \label{thm:categorified special case}
Continue with the notation of Theorem~\ref{thm:main comparison}.  Let $\cF$ be a vector bundle on $\PP^n$ and assume that all summands in the Boij--S\"oderberg decomposition of $\cF$ come from $W^\perp$, i.e. assume that we have an expression
\[
\gamma(\cF) = \sum_{i=0}^n a_i\gamma(\cE_i) \text{ where } a_i\in \QQ_{\geq 0}.
\]
Then $\gamma(\Phi^W_2(\cF)) = N_W \gamma(\cF)$, and the transformed bundle $\Phi^W_2(\cF)$ has a filtration  
\[
0=\cF_{-1} \subseteq \cF_0\subseteq \cF_1\subseteq \dots \subseteq \cF_n=\Phi^W_2(\cF)
\]
where
\[
\cF_i/\cF_{i-1} \cong \cE_i^{\oplus N_W\cdot a_i}.
\]
\end{thm}

We view Theorem~\ref{thm:categorified special case} as a proof of concept for the idea that the positive rational Boij--S\"oderberg decompositions of cohomology tables should admit meaningful categorifications. Namely, the Fourier--Mukai transform $\Phi^W_2$ provides a novel mechanism for introducing a scalar multiple and clearing denominators,
and the spectral sequence for $\Phi^W_2$ from~Corollary~\ref{cor:categorification} is based on supernatural bundles.  
In fact, Theorem~\ref{thm:categorified special case} yields the first known categorification of Boij--S\"oderberg decompositions that genuinely make use of the $\QQ$-coefficients.  

Even the simplest case of Theorem~~\ref{thm:categorified special case} is surprising, at least to the authors.  For instance, as remarked earlier, supernatural bundles with a given root sequence can have nontrivial moduli~\cite[\S6]{ES-river}.  This moduli collapses entirely after applying $\Phi^W_2$:

\begin{cor}\label{cor:supernatural becomes equivariant}
Let $\cF$ be any supernatural sheaf with root sequence $\{-w_0,-w_1,\dots,-w_n\} \setminus \{-w_i\}$.  Then $\Phi^W_2(\cF)$ is a direct sum of copies of the equivariant supernatural sheaf $\cE_i$. More precisely, setting $m=\frac{N_W\cdot \rank \cF}{\rank \cE_i}$, we have $\Phi^W_2(\cF) \cong \cE_i^{\oplus m}$.
\end{cor}

Theorem~\ref{thm:categorified special case} does not apply to a general vector bundle, as for many vector bundles, the summands in the Boij--S\"oderberg decomposition will not come from a single $W^\perp$.  See Remark~\ref{rmk:general} for brief comments about some of the challenges in generalizing Theorem~\ref{thm:categorified special case}.

\subsection*{Acknowledgements} 
We thank David Eisenbud and Frank-Olaf Schreyer for discussions that inspired this work.  We also thank \.{I}zzet Co\c{s}kun, Brendan Hassett, Jack Huizenga, and Matthew Woolf for useful discussions.  We thank the referee for helpful comments.

\section{Background}\label{sec:background}

Fix a field $\bk$ of characteristic zero. 

The Schur functor $\bS_\lambda$ is defined for any partition $\lambda = (\lambda_1 \ge \lambda_2 \ge \cdots \ge 0)$. It can be applied to any vector bundle $\cE$; the result is a vector bundle $\bS_\lambda \cE$. If $\alpha = (\alpha_1 \ge \dots \ge \alpha_r)$ is a weakly decreasing sequence and $\rank \cE = r$, then define $\bS_\alpha \cE = (\det \cE)^{\otimes \alpha_r} \otimes \bS_\lambda \cE$ where $\lambda_i = \alpha_i - \alpha_r$. If $\alpha_r \ge 0$, this is consistent with properties of Schur functors. An important property is that $(\bS_\lambda \cE)^* \cong \bS_\beta \cE$ where $\beta = (-\alpha_r, -\alpha_{r-1}, \dots, -\alpha_1)$. We refer the reader to \cite[\S 2]{weyman} for details on Schur functors, but we point out that our notation for Schur functors coincides with the notation for Weyl functors, denoted ${\bf K}_\lambda$, used there.

Let $S=\bk[x_0,\dots,x_n]$ and $\PP^n=\PP^n_{\bk}$.  A \defi{degree sequence of length $s$} is a sequence $\bd = (d_0,d_1,\dots,d_{s})\in\ZZ^{s+1}$ where $d_i<d_{i+1}$ for each $i$.  A \defi{root sequence} for $\PP^n$ is a sequence $f=(f_1,\dots,f_n)$ where $f_i>f_{i+1}$ for each $i$.  Although there are more general notions of degree sequence and root sequence in the literature on Boij--S\"oderberg theory (\cite[Definition 1]{boij-sod-1} \cite[Introduction]{ES2010}, \cite[Definition 1.2]{eisenbud-erman-categorified}, \cite[Definition 3.1]{kummini-sam}), we will not use those notions.

We compare root sequences via the partial order $f\leq f'$, which holds if $f_i\leq f_i'$ for all $i$, and we define a chain $(\dots, f^{(j)},f^{(j+1)},\dots)_{j \in \ZZ}$ of root sequences as a collection of root sequences such that $f^{(j)}< f^{(j+1)}$ for all $j\in \ZZ$.  

Given a root sequence $f=(f_1,\dots,f_n)$, a vector bundle $\cE$ on $\PP^n$ is a \defi{supernatural bundle of type $f$} if for each $j\in \ZZ$ we have that $\rH^i(\PP^n,\cE(j))\ne 0$ for at most one $i$, and if $\rH^\bullet(\PP^n, \cE(f_j)) = 0$ for $j=1,\dots,n$. These are the vector bundles with the fewest possible nonzero cohomology groups. Eisenbud and Schreyer have shown the existence of supernatural vector bundles via two separate constructions.  The first construction works only in characteristic zero (and was observed by Weyman): if we define a partition $\mu$ via $\mu_i = f_1 - f_{n+1-i} - n + i$, then $\bS_\mu \cQ$ is a supernatural bundle of type $f$~\cite[Theorem~6.2]{ES2008}.  A second construction involves the pushforward of a line bundle from a product of projective spaces, and that construction works in arbitrary characteristic~\cite[Theorem~6.1]{ES2008} (see also \cite{tensor-complexes}). 

Given a degree sequence $\bd=(d_0,\dots,d_s)$, a free complex $F_\bullet = [F_0\gets F_1\gets \cdots \gets F_s\gets 0]$ of graded $S$-modules is a \defi{pure resolution of type $\bd$} if $F_\bullet$ is acyclic and if $F_i\cong S(-d_i)^{b_i}$ for some $b_i>0$.  Boij and S\"oderberg conjectured that for any degree sequence $\bd$, there is a pure resolution of type $\bd$ that, moreover, resolves a Cohen--Macaulay module~\cite{boij-sod-1}.  This conjecture was proven by Eisenbud, Fl{\o}ystad, and Weyman in~\cite[Theorems~0.1 and 0.2]{EFW} in characteristic zero, and by Eisenbud and Schreyer~\cite[Theorem~0.1]{ES2008} in arbitrary characteristic. We will only be interested in the case $s=n$.

The construction of pure resolutions from \cite[Theorem~0.1]{EFW} will be most relevant for us, and we discuss a relative version of that construction. Let $E$ be a vector space of dimension $n$ over $\bk$. Let $A = \Sym(E)$.  

Let $\bd = (d_0, \dots, d_n)$ be a degree sequence. For $j=0,1,\dots,n$, define partitions $\lambda(\bd)^j = (\lambda(\bd)^j_1, \dots, \lambda(\bd)^j_n)$ by 
\begin{equation}\label{eqn:lambda i}
\lambda(\bd)^j_i = \begin{cases} 
d_n - d_{i-1} - (n-i) & i \le j,\\
d_n - d_i - (n-i) & i > j.
\end{cases}
\end{equation}
If it is clear from context, we will write $\lambda^j$ instead of $\lambda(\bd)^j$. The \defi{EFW complex} of $\bd$ and $E$ is the following complex
\begin{subequations}
\begin{align} \label{eqn:EFW-local}
A \otimes \bS_{\lambda(\bd)^0}(E)\gets A \otimes \bS_{\lambda(\bd)^1}(E)  \gets \cdots \gets A \otimes \bS_{\lambda(\bd)^n}(E) \gets 0.
\end{align}
It is an acyclic complex and resolves a finite length Cohen--Macaulay module.  Furthermore, this complex is equivariant for the action of $\GL(E)$. In fact, the differentials $\bS_{\lambda(\bd)^i}(E) \to A \otimes \bS_{\lambda(\bd)^{i-1}}(E)$ are defined using Pieri's rule, which gives an inclusion of representations
\[
\bS_{\lambda(\bd)^i}(E) \to A_{d_i-d_{i-1}} \otimes \bS_{\lambda(\bd)^{i-1}}(E).
\]
The differential is unique up to scalar because the target representation is multiplicity-free.

These facts remain true if we replace $E$ by a vector bundle $\cE$, so we can globalize the construction as follows. Let $\cE$ be a vector bundle of rank $n$ over a $\bk$-scheme $X$. Let $\cA = \Sym(\cE)$.  
The \defi{EFW complex} of $\bd$ and $\cE$ is the following complex of vector bundles 
\begin{align} \label{eqn:EFW}
\cA \otimes \bS_{\lambda(\bd)^0}(\cE)\gets \cA \otimes \bS_{\lambda(\bd)^1}(\cE)  \gets \cdots \gets \cA \otimes \bS_{\lambda(\bd)^n}(\cE) \gets 0.
\end{align}
\end{subequations}
It is an acyclic complex and resolves a Cohen--Macaulay sheaf which is set-theoretically supported in the zero section of $\cE^* = \Spec_X(\cA)$ (since locally our complex is modeled on \eqref{eqn:EFW-local}). Furthermore, in the polynomial case above, the cokernel has a grading. In the global setting, this grading defines a filtration for the cokernel whose associated graded is a vector bundle on the zero section of $\cE^*$.

Let $V$ be a vector space of dimension $n+1$ and let $\PP(V)$ be the space of lines. Then we have a tautological exact sequence
\[
0 \to \cO(-1) \to V\otimes \cO \to \cQ \to 0.
\]
Since we will use it a few times, we recall the Borel--Weil--Bott theorem for $\PP(V)$ (see \cite[Corollary 4.1.9]{weyman}). Given a permutation $\sigma$, we define the {\bf length} of $\sigma$ to be $\ell(\sigma) = \#\{ i < j \mid \sigma(i) > \sigma(j) \}$. Also, define $\rho = (n, n-1, \dots, 1, 0)$. Given a sequence of integers $\alpha \in \ZZ^{n+1}$, we define $\sigma \bullet \alpha = \sigma(\alpha + \rho) - \rho$.

\begin{thm}[Borel--Weil--Bott] \label{thm:bott} 
Let $\alpha = (\alpha_1, \dots, \alpha_n)$ be a weakly decreasing sequence, pick $d \in \ZZ$, and set $\beta = (d, \alpha_1, \dots, \alpha_n)$. Then exactly one of the following two situations occurs.
\begin{enumerate}[\rm 1.]
\item There exists $\sigma \ne \mathrm{id}$ such that $\sigma \bullet \beta = \beta$. Then all cohomology of $\bS_\alpha \cQ^* \otimes \cO(d)$ vanishes.
\item There is a $($unique$)$ $\sigma$ such that $\gamma = \sigma \bullet \beta$ is a weakly decreasing sequence. Then
\[
\rH^{\ell(\sigma)}(\PP(V), \bS_\alpha \cQ^* \otimes \cO(d)) = \bS_\gamma(V^*)
\]
and all other cohomology vanishes. 
\end{enumerate}
\end{thm}

\section{Proof of Theorem~\ref{thm:main comparison}}

Pick a vector space $V$ of dimension $n+1$ and let $X = \PP(V) \times \PP(V)$. We have a $\GL(V)$-equivariant isomorphism $\rH^0(X; \cO(1) \boxtimes \cQ) = {\rm End}(V)$ which can be explicitly obtained by sending $\phi \in {\rm End}(V)$ to the composition 
\[
\cO(-1) \boxtimes \cO \to V \boxtimes \cO \xrightarrow{\phi} \cO \boxtimes V \to \cO \boxtimes \cQ.
\]
So the zero locus of the identity map in ${\rm End}(V)$ is the diagonal $\Delta_{\PP(V)} \subset X$.

Take $\cE = \cO(-1) \boxtimes \cQ^*$ and build \eqref{eqn:EFW} with $\bd = {\bf w} = (w_0, \dots, w_n)$. Under the identification $\rH^0(X; \cE^*) = {\rm End}(V)$ above, the section of $\cE^*$ corresponding to the identity in ${\rm End}(V)$ gives a map $\Sym(\cE) \to \cO_X$; tensoring with it, we get an equivariant complex:
\[
\cO(-|\lambda^0|) \boxtimes \bS_{\lambda^0}(\cQ^*) \gets
\cO(-|\lambda^1|) \boxtimes \bS_{\lambda^1}(\cQ^*) \gets 
\cdots \gets
\cO(-|\lambda^n|) \boxtimes \bS_{\lambda^n}(\cQ^*) \gets 0.
\]
We define $\bE_W$ as the complex obtained by twisting the above complex by $\cO(|\lambda^0|-w_0) \boxtimes \cO(w_n - n)$. 

\begin{lemma}
$\bE_W$ is acyclic and it resolves a Cohen--Macaulay sheaf $\cU_W$ which is set-theoretically supported on the diagonal $\Delta_{\PP(V)}$.
Furthermore, $\cU_W$ has a filtration whose associated graded is a vector bundle over $\Delta_{\PP(V)}$.
\end{lemma}

\begin{proof}
To check that $\bE_W$ is acyclic, it suffices to check this statement locally, so we may assume we are working over a local ring. In that case, we apply the Eagon--Northcott generic perfection theorem \cite[Theorem 1.2.14]{weyman}. Since $\Spec(\Sym(\cE))$ is smooth, it is in particular Cohen--Macaulay, as are all of its local rings, so the depth of an ideal coincides with its codimension. Let $M$ be the cokernel of \eqref{eqn:EFW}. Then $M$ is locally a perfect module by the Auslander--Buchsbaum formula \cite[Theorem 19.9]{eisenbud}. So the criteria in \cite[Theorem 1.2.14]{weyman} are asking for the codimension of ${\rm Ann}(M)$ (which is $n$) to be $n$ after tensoring with $\cO_X$. Now, ${\rm Ann}(M \otimes \cO_X)$ and ${\rm Ann}(M) \cO_X$ have the same radical, and the radical of ${\rm Ann}(M)$ is the zero section of $\Spec(\Sym(\cE))$, so ${\rm Ann}(M) \cO_X$ is, up to radical, the diagonal of $X$, so this is just the statement that the diagonal has codimension $n$.  

The filtration in the last sentence is obtained by tensoring the filtration on the cokernel of \eqref{eqn:EFW} discussed above along the map $\Sym(\cE) \to \cO_X$, noting that tensoring is compatible with taking associated graded since the associated graded is flat over the zero section. The associated graded is Cohen--Macaulay (which can be deduced from the generic perfection theorem) and hence is a vector bundle since its support is smooth (the implication is provided by the Auslander--Buchsbaum formula).
\end{proof}

The sheaf $\cU_W$ has an action of $\GL(V)$, so its pushforward along either projection $\PP(V) \times \PP(V) \to \PP(V)$ is homogeneous, and hence is a vector bundle. Furthermore, the pushforward from the support of $\cU_W$ to either $\PP(V)$ is exact, so we can calculate the rank of this bundle as the multiplicity of the cokernel of the EFW complex, which is $N_W$ (this is implicit in the description of the cokernel given in \cite[Theorem 3.2(2)]{EFW}, and follows more directly from the description given in \cite[Remark 4.3.3]{symc1}).

By Borel--Weil--Bott (Theorem~\ref{thm:bott}), the sheaf cohomology of $\cO(d) \otimes \bS_\mu(\cQ^*)$ vanishes if and only if $d = \mu_i - i$ for some $i$, so $\bS_\mu(\cQ^*)$ is a supernatural vector bundle whose Hilbert polynomial has roots $(\mu_1 - 1, \mu_2 - 2, \dots, \mu_n - n)$. In particular, $\bS_{\lambda^j}(\cQ^*) \otimes \cO(w_n - n)$ is a supernatural vector bundle whose Hilbert polynomial has roots $\{-w_0,\dots,-w_n\}\setminus \{-w_j\}$.  Thus, setting 
\begin{align} \label{eqn:Wperp-defn}
W^\perp_j = \cE_j := \bS_{\lambda({\bf w})^j}(\cQ^*) \otimes \cO(w_n - n)
\end{align}
we obtain the desired supernatural bundle. This proves (3). It follows from Borel--Weil--Bott (Theorem~\ref{thm:bott}) that $\rH^i(\PP(V), \cE_i(-w_i)) = \bS_{\mu(W)}(V^*)$ and that the cohomology vanishes in other degrees. This proves part of (2).

We now use the cohomology calculation to deduce that $W$ and $W^\perp$ both span a subgroup of $\rK_0(\PP^n)$ of index $N_W$. We have a nondegenerate bilinear pairing on $\rK_0(\PP^n)$ given by 
\[
\langle [\cE], [\cF] \rangle = \sum_{i=0}^n (-1)^i \dim \Ext^i(\cE, \cF)
\]
(nondegeneracy can be proven by observing that on the basis $[\cO], [\cO(-1)], \dots, [\cO(-n)]$, the Gram matrix is upper unitriangular). First, we show that $W$ is linearly independent in $\rK_0(\PP^n)_\QQ$. Suppose that we have an expression $0 = \sum_i a_i [\cO(-w_i)]$. Applying $\langle [\cE_j^*] , - \rangle$, we get:
\begin{align*}
0&=\sum_{i=0}^n a_i \langle  [\cE_j^*] , [\cO(-w_i)] \rangle\\
& = \sum_{i=0}^n a_i \left(\sum_{\ell =0}^n (-1)^\ell \dim \Ext^\ell(\cE_j^*, \cO(-w_i))\right)\\
&=\sum_{i=0}^n a_i \left(\sum_{\ell =0}^n (-1)^\ell \dim \rH^\ell(\cE_j(-w_i))\right).\\
\intertext{Using the above cohomology computations for $\cE_j$ then yields}
&=a_j \left((-1)^j\dim \rH^j(\cE_j(-w_j))\right) = (-1)^j a_jN_W.
\end{align*}
It follows that $a_j=0$ for all $j$, and thus that $W$ is linearly independent in $\rK_0(\PP^n)_\QQ$. 
A similar argument shows that $W^\perp$ is linearly independent.  Thus $W$ and $W^\perp$ each forms a basis for $\rK_0(\PP^n)_\QQ$.

Now suppose that we have an expression $[\cF] = \sum_i a_i [\cO(-w_i)]$ where $a_i \in \QQ$ and $\cF$ is a vector bundle. Again, we deduce that 
$a_j = \frac{(-1)^j}{N_W} \langle [\cE_j^*], [\cF] \rangle$, so $N_W [\cF]$ is in the $\ZZ$-span of $W$ for all $\cF$. A similar argument applies to $W^\perp$. This finishes the proof of (1) and (2).

\begin{remark}
Alternatively, to prove that $W$ spans a subgroup of index $N_W$ in $\rK_0(\PP^n)$, it suffices to prove that for any $\cO(-d)\notin W$, we can write $N_W\cdot [\cO(-d)]$ as a $\ZZ$-linear combination of the classes $[\cO(-w_i)]$.  There is a unique $j$ such that $w_j<d<w_{j+1}$ and thus ${\bf e} := (w_0,w_1,\dots,w_j,d,w_{j+1},\dots,w_n)\in \ZZ^{n+2}$ is a degree sequence.  Using the construction of the EFW complex in the previous section, there is a pure resolution 
\begin{align*}
F_\bullet = [S(-w_0)^{b_0}\gets \dots\gets S(-w_j)^{b_j} \gets S(-d)^{b_{j+1}} \gets S(-w_{j+1})^{b_{j+2}}\gets \dots \gets S(-w_{n+1})^{b_{n+2}}\gets 0]
\end{align*}
of type ${\bf e}$ that resolves a finite length module. The corresponding complex of sheaves $\widetilde{F}_\bullet$ on $\PP^n$ is thus exact. From \eqref{eqn:lambda i}, we see that $b_{j+1} = \dim \bS_{\mu(W)}(\bk^{n+1})$, so $b_{j+1} = N_W$. It follows that $N_W\cdot [\cO(-d)]$ can be written as a $\ZZ$-linear combination of the classes $[\cO(-w_i)]$ in $\rK_0(\PP^n)$.
\end{remark}

\section{Properties of $\cU_W$ and the proof of Corollary~\ref{cor:categorification}} \label{sec:UW}

Recall that we use $\cU_W$ to denote the sheaf resolved by $\bE_W$.  Given a finite map $f \colon X\to \PP^n$ we say that a coherent sheaf $\cF$ on $X$ is an \defi{Ulrich sheaf for $f$} if $f_*\cF\cong \cO_{\PP^n}^N$ for some $N$.  

\begin{defn}
Let $\Delta_W$ be the scheme-theoretic support of $\cU_W$.
\end{defn}

\begin{prop}\label{prop:ulrich}
For $i=0,\dots,n$, $\cU_W(w_i,-w_i)$ is an Ulrich sheaf for $p_1\colon \Delta_W\to \PP^n$. More precisely, $p_{1*} (\cU_W(w_i,-w_i)) \cong \cO_{\PP^n}^{N_W}$.
\end{prop}

\begin{proof}
Since $p_1|_{\Delta_W}$ is affine, we have $\rR^k p_{1*} (\cU_W(w_i,-w_i))=0$ for $k>0$.  
Also, since $-w_i$ is a root of each supernatural bundle $\cE_j$ for $j \ne i$, the hypercohomology spectral sequence for computing $\rR p_{1*} (\cU_W(w_i,-w_i)) = \rR p_{1*}(\bE_W(w_i,-w_i))$ consists of a single term $\rE^1_{-i,i}=\cO \boxtimes \rH^i(\PP^n,\cE_i(-w_i)) \cong \cO_{\PP^n}^{N_W}$ (the isomorphism follows from Theorem~\ref{thm:main comparison}(2)).   Hence the spectral sequence immediately degenerates yielding $p_{1*} (\cU_W(w_i,-w_i)) \cong \cO_{\PP^n}^{N_W}$.
\end{proof}

\begin{proof}[Proof of Corollary~\ref{cor:categorification}]
We first prove part (3).  For $i=1$ we have
\[
p_2^*\cF \otimes \bE_W = [\cO(-w_0)\boxtimes \left(\cE_0\otimes \cF\right) \gets \cdots \gets  \cO(-w_n)\boxtimes \left(\cE_n\otimes \cF\right) \gets 0].
\]
To compute $\Phi_1(\cF)$, we apply $\rR p_{1*}$ to this complex.  The hypercohomology spectral sequence for $\rR p_{1*}$ has the form:
\[
\rE^1_{p,q} = \rR^{-p}p_{1*} \left( \cO(-w_q)\boxtimes \cE_q\otimes \cF\right) \Longrightarrow \Phi^W_1(\cF).
\]
Using the fact that $\rR^{-p}p_{1*} \left( \cO(-w_q)\boxtimes \cE_q\otimes \cF\right)\cong \cO(-w_q)\otimes \rH^{-p}(\PP^n, \cE_q\otimes \cF)$, we obtain the statement for $i=1$.  The proof for $i=2$ is similar.

We next prove part (1).  It suffices to prove the statement for $\rK_0(\PP^n)_{\QQ}$.  Proposition~\ref{prop:ulrich} shows that $\Phi^W_1(\cO(-w_i)) = \cO(-w_i)^{N_W}$ for $i=0,\dots,n$, so we see that $\rK_0(\Phi^W_1)$ acts as multiplication by $N_W$ on each $[\cO(-w_i)]$.  By Theorem~\ref{thm:main comparison}(1), the classes of the $\cO(-w_i)$ form a basis of $\rK_0(\PP^n)_{\QQ}$, and it follows that $\Phi^W_1$ acts as multiplication by $N_W$ on all of $\rK_0(\PP^n)_{\QQ}$.

For $\Phi^W_2$, we first observe that the $\widehat{\rE}^1$-page of the spectral sequence for $\Phi^W_2(\cE_i)$ has a single nonzero term $\widehat{\rE}^1_{-i,i}=\rH^i(\PP^n, \cE_i(-w_i)) \otimes \cE_i $.  By Theorem~\ref{thm:main comparison}(2), this is isomorphic to $\cE_i^{\oplus N_W}$, and thus $[\Phi^W_2(\cE_i)]=N_W\cdot [\cE_i]$ for $i=0,1,\dots, n$.  Since the $\cE_i$ form a basis of $\rK_0(\PP^n)_{\QQ}$ by Theorem~\ref{thm:main comparison}(2), it follows that $\Phi^W_2$ acts as multiplication by $N_W$ on all of $\rK_0(\PP^n)_{\QQ}$.

For part (2), we observe that since $W$ and $W^\perp$ are bases of $\rK_0(\PP^n)_{\QQ}$, the coefficients will be unique.  By part (1) of this theorem it suffices to write down an expression for the class of $\Phi^W_i(\cF)$.  In $\rK_0(\PP^n)$, the alternating sum of the terms in a spectral sequence is invariant under turning the page of a spectral sequence.  Thus, since the spectral sequence $\rE^1_{p,q}$ abuts to $\Phi_1^W(\cF)$, we have
\begin{align*}
[\Phi_1^W(\cF)] &= \sum_{p,q} (-1)^{p+q} [\rE^1_{p,q}] \\
&= \sum_{q=0}^n \sum_{p=-n}^0 (-1)^{p+q} \rh^{-p}(\PP^n,\cE_q\otimes \cF)\cdot [\cO(-w_q)] \\
&= \sum_{q=0}^n \chi(\PP^n,\cE_q\otimes \cF)\cdot [\cO(-w_q)].
\end{align*}
Since $[\Phi_1^W(\cF)] = \frac{1}{N_W}\cdot [\cF]$, this yields the statement for $W$.  The statement for $W^\perp$ is similar.
\end{proof}

\begin{example}\label{ex:023 line}
Returning to Example~\ref{ex:023}, let $L$ be a line in $\PP^2$.  Then in $\rK_0(\PP^2)_{\QQ}$ we have
\[
[\cO_{L}] = \tfrac{2}{3}[\cO_{\PP^2}] - [\cO_{\PP^2}(-2)]+\tfrac{1}{3}[\cO_{\PP^2}(-3)].
\]
Applying the Fourier--Mukai transform $\Phi^W_1$ categorifies this decomposition, as if we use the spectral sequence $\rE^1_{p,q}$ from Corollary~\ref{cor:categorification}, then we obtain an exact sequence:
\[
0\gets \Phi^W_1(\cO_L) \gets \cO^2 \gets \cO(-2)^3 \gets \cO(-3)\gets 0.
\]

Using $W^\perp$ on the other hand, we have the decomposition
\[
[\cO_{L}] = \tfrac{1}{3}[\cE_0]+\tfrac{1}{3}[\cE_1] - \tfrac{2}{3}[\cE_2],
\]
and a direct computation using the other spectral sequence (plus the fact that $\Ext^1(\cE_1,\cE_0)=\rH^1(\Sym^2\cQ)=0$) yields an exact sequence:
\[
0\gets \Phi^W_2(\cO_L) \gets \cE_0\oplus \cE_1\gets\cE_2^2\gets 0. \qedhere
\]
\end{example}

\begin{prop}
Let $\cF$ be a coherent sheaf on $\PP^n$. Then for any $W$ and $i \in \{1,2\}$, $\Phi_i^W(\cF)$ is a coherent sheaf concentrated in cohomological degree $0$.
\end{prop}

\begin{proof}
We have $\Phi_1^W(\cF) = \rR p_{1*}( {\rm L} p_2^* \cF \otimes^{\rm L} \cU_W)$. Note that if we have a short exact sequence $0 \to M \to \cU_W \to \cU_W/M$, then we get an exact triangle
\[
\rR p_{1*}( {\rm L} p_2^* \cF \otimes^{\rm L} M) \to \rR p_{1*}( {\rm L} p_2^* \cF \otimes^{\rm L} \cU_W) \to \rR p_{1*}( {\rm L} p_2^* \cF \otimes^{\rm L} \cU_W/M) \to.
\]
So to show that the middle term is a coherent sheaf in cohomological degree $0$, it suffices to prove this for the outer two terms. Since $\cU_W$ has a filtration whose associated graded is a vector bundle on the diagonal $\Delta_{\PP^n}$, it suffices to handle that case. So suppose $M$ is a vector bundle on $\Delta_{\PP^n}$ and let $\iota \colon \PP^n \to \Delta_{\PP^n}$ be the isomorphism $x \mapsto (x,x)$. Then we can identify $M$ with a vector bundle on $\PP^n$ and we write $\iota_* M$ in place of $M$. Now we have
\begin{align*}
\rR p_{1*} ( {\rm L} p_2^* \cF \otimes^{\rm L} \iota_* M) 
&= \rR p_{1*} \rR \iota_* ( {\rm L} \iota^* {\rm L} p_2^* \cF \otimes^{\rm L} M) &\text{(projection formula)}\\
&= \cF \otimes^{\rm L} M & (p_1 \circ \iota = p_2 \circ \iota = {\rm id})\\
&= \cF \otimes M & \text{($M$ is flat)}
\end{align*}
and so we conclude that $\Phi_1^W(\cF)$ is a coherent sheaf in cohomological degree $0$. The proof for $\Phi_2^W(\cF)$ is exactly the same.
\end{proof}

\section{The effect of $\Phi^W_2$ on cohomology tables}

In the situation of Theorem~\ref{thm:categorified special case}, we need to prove that $\gamma(\Phi^W_2(\cF)) = N_W \gamma(\cF)$. Although $\Phi^W_2$ acts as scalar multiplication in the Grothendieck group, it is not true that it always acts as scalar multiplication on cohomology tables; see Example~\ref{ex:023 and Ulrich}. However, we do have the following semicontinuity statement in general.

\begin{prop}\label{prop:semicont}
Let $\cF$ be any coherent sheaf on $\PP^n$. Then for all $i$ and all $d$, we have
\[
\rh^i(\PP^n, \Phi^W_2(\cF)(d)) \geq N_W\cdot \rh^i(\PP^n, \cF(d)).
\]
In the following $3$ cases, the above is an equality:
\begin{enumerate}[\indent \rm (1)]
\item $d \in \{-w_0, \dots, -w_n\}$,
\item $d > -w_0$, or
\item $d < -w_n$.
\end{enumerate}
\end{prop}

A semicontinuity result is the best we can hope for in full generality, in light of examples like the following.

\begin{example}\label{ex:023 and Ulrich}
Continue with the notation of Example~\ref{ex:023}. Using the spectral sequence $\widehat{\rE}^1_{p,q}$  from Corollary~\ref{cor:categorification} to compute $\rR p_{2*}\mathcal U_W=p_{2*}\mathcal U_W$, we get an extension:
\[
0 \to \cO(1) \to p_{2*} \cU_W \to \Omega^1(1) \to 0.
\]
But $\Ext^1(\Omega^1(1),\cO(1))=\rH^1(\PP^2, (\Omega^1)^*)=0$ and so $p_{2*} \mathcal U_W=\Omega^1(1)\oplus \cO(1)$.  This has the same Hilbert polynomial as $\cO^3$, but $\mathcal U_W$ is not an Ulrich sheaf for $p_2$.  Moreover, the cohomology table of $\Phi^W_2(\cO)$ does not equal a scalar multiple of the cohomology table of $\cO^3$.
\end{example}

Our proof of Proposition~\ref{prop:semicont} relies on a more detailed analysis of the equivariant extensions---or lack thereof---between certain equivariant bundles.  We motivate this discussion with an example.

\begin{example}
Recall our notation for the tautological exact sequence $0 \to \cO(-1) \to V \otimes \cO \to \cQ \to 0$. The equivariant bundle $\Sym^2(V) \otimes \cO_{\PP(V)}$ has a filtration whose quotients are $\cO(-2)$, $\cO(-1) \otimes \cQ$, and $\Sym^2(\cQ)$. There are two proper equivariant subbundles: $\cO(-2)$ and the kernel $\cE$ of the surjection $\Sym^2(V) \to \Sym^2(\cQ)$.

We first consider the subbundle $\cO(-2)$, which determines an exact sequence:
\[
0\to \cO(-2)\to \Sym^2(V) \otimes \cO_{\PP(V)} \to \cW\to 0.
\]
We claim that if you reverse the roles of the subbundle and the quotient bundle in the above extension, then there will be no nontrivial extensions.  In other words, we claim that $\Ext^1(\cO(-2),\cW)=0$; this follows by considering the long exact sequence:
\[
\cdots \to \Ext^1(\cO(-2), \Sym^2(V)) \to \Ext^1(\cO(-2), \cW) \to \Ext^2(\cO(-2), \cO(-2)) \to \cdots,
\]
since the outer two terms are zero. 

If we consider the other subbundle $\cE$ then we also have an exact sequence
\[
0\to \cE\to  \Sym^2(V) \otimes \cO_{\PP(V)} \to \Sym^2(\cQ)\to 0.
\]
If we reverse the roles of the subbundle and the quotient bundle, then we get $\Ext^1(\cE, \Sym^2(\cQ)) = 0$, which follows by a similar argument.
\end{example}

The splitting observed in the above example holds in general, at least if we restrict our attention to equivariant extensions.

\begin{lemma} \label{lem:ext-vanishing}
Pick a partition $\lambda$. Let $\cE \subset \bS_\lambda(V) \otimes \cO_{\PP(V)}$ be an equivariant subbundle and let $\cF$ be the quotient bundle. Every $\GL(V)$-equivariant extension
\[
0\to \cF \to \cU \to \cE\to 0
\]
splits, and thus  $\Ext^1_{\PP(V)}(\cE, \cF)^{\GL(V)} = 0$.
\end{lemma}

\begin{proof}
Pick a direct sum decomposition $V = L \oplus W$ where $\dim L = 1$. Let $\fp$ be the Lie algebra of the stabilizer subgroup in $\GL(V)$ of the point $[L] \in \PP(V)$, so $\fp \cong (\fgl(W) \times \fgl(L)) \ltimes (L \otimes W^*)$. Let $E$ and $F$ be the fibers of $\cE$ and $\cF$ over $[L]$. The assignment $\cE \mapsto E$ is an equivalence between the category of homogeneous bundles on $\PP(V)$ and the category of $\fp$-modules. It then suffices to show that $\Ext^1_\fp(E, F) = 0$. Furthermore, a $\fp$-module is the same as a $\Sym(L \otimes W^*)$-module with a compatible action of $\fgl(L) \times \fgl(W)$. We ignore $\fgl(L)$ since it only keeps track of the grading.

By $\fgl(W)$-equivariance, both $E$ and $F$ are a direct sum of Schur functors on $W$. Furthermore, $\bS_\lambda(V)$ has the property that the submodule generated by any $\bS_\mu(W)$ contains all $\bS_\nu(W)$ where $\nu \subseteq \mu$: $\bS_\lambda(V)$ is the graded dual (as a $\Sym(W)$-module) of the cokernel of an EFW complex over $\Sym(W)$; this cokernel is the quotient of a module of the form $\Sym(W) \otimes \bS_\alpha(W)$, and these modules have this property \cite[Lemma 1.6]{sam-weyman}.

In particular, $E$ has the same property, and so there is no nonzero equivariant map $E \otimes \Sym(L \otimes W^*) \to F$. We mentioned above that a $\fp$-module is the same as a $\Sym(L \otimes W^*)$ with a compatible action of $\fgl(L) \times \fgl(W)$, and so this means that given any  equivariant extension of $\fp$-modules of the form $0 \to F \to \eta \to E \to 0$, we can split it as a $\fp$-module by splitting it as a sequence of $\fgl(L) \times \fgl(W)$-modules since the product of any $\fgl(L) \times \fgl(W)$-equivariant lifting of $E$ with the ideal generated by $L \otimes W^*$ cannot intersect $F$ (splitting as $\fgl(L) \times \fgl(W)$-modules is possible since $\fgl(L) \times \fgl(W)$ is a reductive Lie algebra).
\end{proof}

\begin{lemma}\label{lem: extension}
Continue with the notation of Theorem~\ref{thm:main comparison}.
Fix any $d\in \ZZ$.  There is a sheaf $\mathcal G$ on $\PP^n\times \AA^1$, flat over $\AA^1$, such that
\[
\cG_x \cong \begin{cases}
\Phi_1^W(\cO(d)) & \text{if } x=0\in \AA^1,\\
\cO(d)^{N_W} & \text{if } x \in \AA^1 \setminus \{0\}.
\end{cases}
\]
Furthermore, $\Phi_1^W(\cO(d)) \cong \cO(d)^{N_W}$ if either: $d \in \{-w_0, \dots, -w_n\}$, $d > -w_0$, or $d < -w_n$.
\end{lemma}

\begin{proof}
\addtocounter{equation}{-1}
\begin{subequations}
Define $\mu(W)$ as in Theorem~\ref{thm:main comparison}. Let $q\colon \PP^n\times \AA^1 \to\PP^n$ be the first projection. If $d = -w_i$ for some $i$, then by the argument in the proof of Corollary~\ref{cor:categorification} we have that $\Phi^W_1(\cO(d))\cong \bS_{\mu(W)}(V) \otimes \cO(d)$ as an equivariant module and we take $\mathcal G:=q^*(\bS_{\mu(W)}(V) \otimes \cO(d))$. 

If $d\notin \{-w_0, \dots, -w_n\}$ then there exists a unique $i$ such that $-w_i>d>-w_{i+1}$ (for notation, $-w_{-1} = \infty$ and $-w_{n+1} = -\infty$). The spectral sequence from Corollary~\ref{cor:categorification}(3) for computing $\Phi^W_1(\cO(d))$ consists of two (possibly empty) strands $A$ and $B$ where 
\begin{align*}
A_j &=
\begin{cases}
 \cO(-w_j)\otimes \rH^i(\PP^n, \cE_j(d)) & \text{ if } j\in [0,i]\\
 0 & \text{ else}
 \end{cases},\\
B_j &=
\begin{cases}
 \cO(-w_j)\otimes \rH^{i+1}(\PP^n, \cE_j(d)) & \text{ if } j\in [i+1, n]\\
 0 & \text{ else}
 \end{cases}.
\end{align*}
The $\rE^2$-page of this spectral sequence yields the $\GL(V)$-equivariant short exact sequence
\begin{equation}\label{eqn:PhiW1 extension}
0\to\ker A \to \Phi^W_1(\cO(d)) \to \coker B \to 0,
\end{equation}
which we will show splits. 

Let $\bF$ be the EFW resolution of type ${\bf e} = (w_0,w_1,\dots,w_i,d,w_{i+1},\dots,w_n)$.  We let $A'$ be the subcomplex of the sheafified complex $\widetilde{\bF}$ consisting of all terms in homological degrees $0,\dots,i$, and we let $B'$ be the quotient complex of $\widetilde{\bF}$ consisting of all terms in homological degrees $i+2,\dots,n$. 

Using \eqref{eqn:lambda i}, we see that $\lambda(\be)^{i+1} = \mu(W)$. Also, if $k \le i$, then 
\[
\lambda(\be)^k_j = \begin{cases} 
w_n - w_{j-1} - (n+1-j) & \text{if } j \le k\\
w_n - w_j - (n+1-j) & \text{if } k+1 \le j \le i\\
w_n - d - (n-i) & \text{if } j = i+1\\
w_n - w_{j-1} - (n+1-j) & \text{ if } j > i+1
\end{cases},
\]
and if $k \ge i+2$, then
\[
\lambda(\be)^k_j = \begin{cases} 
w_n - w_{j-1} - (n+1-j) & \text{if } j \le i+1\\
w_n - d - (n-i-1) & \text{if } j = i+2\\
w_n - w_{j-2} - (n+1-j) & \text{if } i+3 \le j \le k\\
w_n - w_{j} - (n+1-j) & \text{ if } j > k+1
\end{cases}.
\]

Recall that $\cE_j = \bS_{\lambda({\bf w})^j}(\cQ^*) \otimes \cO(w_n-n)$ where ${\bf w} = (w_0, \dots, w_n)$. It follows from Borel--Weil--Bott (Theorem~\ref{thm:bott}) that $A'_k \cong \cO(-w_k) \otimes \rH^i(\PP^n, \cE_k(d))$ for $k = 0,\dots,i$, and that $B'_k \cong \cO(-w_k) \otimes \rH^{i+1}(\PP^n, \cE_k(d))$ for $k=i+2, \dots, n$. Both complexes are $\GL(V)$-equivariant, and the differentials in the EFW complex are uniquely determined by the condition that they are $\GL(V)$-equivariant (a consequence of Pieri's rule, see for example \cite[\S 1.2]{sam-weyman}), so we deduce that $A'\cong A$ and $B'\cong B$.

In particular, this implies that we have $0\to B'\to \bS_{\mu(W)}(V) \otimes \cO(d) \to A'\to 0$, and since $B\cong B'$ and $A\cong A'$, there is an extension coming from $\widetilde{\bF}$:
\begin{equation}\label{eqn:EFW extension}
0\to \coker B \to \bS_{\mu(W)}(V) \otimes \cO(d)\to \ker A\to 0.
\end{equation}
Note that the positions of $\ker A$ and $\coker B$ are reversed in \eqref{eqn:PhiW1 extension} and \eqref{eqn:EFW extension}.  

Since \eqref{eqn:PhiW1 extension} and \eqref{eqn:EFW extension} are both equivariant extensions, we may apply Lemma~\ref{lem:ext-vanishing} to conclude that the sequence in \eqref{eqn:PhiW1 extension} splits. Let $\eta \in \Ext^1(\ker A, \coker B)$ be the class of the extension \eqref{eqn:EFW extension}. We let $\mathcal G$ be the family of bundles on $\PP^n\times \AA^1$ where over $t \in \AA^1$, the bundle $\mathcal G_t$ is the extension corresponding to $t \eta$.  This bundle satisfies the conditions of the lemma.

Finally, we note that if $d > -w_0$, then $A_\bullet = 0$ and $\ker A = 0$, so $\Phi^W_1(\cO(d)) \cong \coker B \cong \bS_{\mu(W)}(V) \otimes \cO(d)$. A similar thing happens when $d < -w_n$ with the roles of $A$ and $B$ reversed.  So in these cases, the bundle $\mathcal G$ is simply $\cO(d)^{N_W}$.
\end{subequations}
\end{proof}

\begin{proof}[Proof of Proposition~\ref{prop:semicont}]
For (1): we compute that
\begin{align*}
\rH^i(\PP^n,\Phi^W_2(\cF)(d)) &= \rH^i(\PP^n, p_{2*}(\mathcal U_W\otimes p_{1}^*\cF)\otimes \cO(d))&\text{(By definition)}\\\
&= \rH^i(\PP^n, p_{2*}(\mathcal U_W\otimes p_{1}^*\cF\otimes \cO(0,d)))&\text{(Projection formula)}\\\
&=\rH^i(\Delta_W,\mathcal U_W\otimes p_1^*\cF\otimes \cO(0,d))&\text{($p_2|_{\Delta_W}$ is affine)}\\
&=\rH^i(\PP^n,p_{1*}\left(\mathcal U_W(0,d)\otimes p_1^*\cF\right))&\text{($p_1|_{\Delta_W}$ is affine)}\\
&=\rH^i(\PP^n,p_{1*}(\mathcal U_W(0,d))\otimes \cF)&\text{(Projection formula)}\\
&=\rH^i(\PP^n, \Phi^W_1(\cO(d)) \otimes \cF). & \text{(By definition)}
\end{align*}
Let $\mathcal G$ be the vector bundle over $\PP^n \times \AA^1$ constructed in Lemma~\ref{lem: extension}. Let $q \colon \PP^n\times \AA^1 \to \PP^n$ be the projection map. Then $\mathcal G\otimes q^*\mathcal F$ is flat over $\AA^1$ with special fiber $\Phi^W_1(\cO(d)) \otimes \mathcal F$ and general fiber $\mathcal F(d)^{\oplus N_W}$ since $(\cG \otimes q^*\cF)_t \cong \cG_t \otimes \cF$. So semicontinuity and the above computation imply:
\[
\rh^i(\PP^n,\Phi^W_2(\cF)(d)) = \rh^i(\PP^n, \Phi^W_1(\cO(d)) \otimes \cF) \geq \rh^i(\PP^n, \cF(d)^{\oplus N_W}). 
\]
When $d \in \{-w_0, \dots, -w_n\}$, or $d > -w_0$, or $d < -w_n$, the sheaf $\cG$ is constant over $\AA^1$, so we get equality in the above formula.
\end{proof}

\section{Application to pure resolutions}

\begin{proof}[Proof of Corollary~\ref{cor:pure resolutions}]
Note that $\bE_W\otimes p_{2}^*\cF$ looks like:
\[
\cO(-w_0)\boxtimes \cE_0\otimes \cF \gets \dots \gets \cO(-w_n)\boxtimes \cE_n\otimes \cF\gets 0.
\]
The Castelnuovo--Mumford regularity of a supernatural sheaf with roots $-f_1\geq -f_2\geq \dots\geq  -f_n$ is $-f_1+1$.  Hence
\[
\reg \cE_i = \begin{cases}
-w_0+1 & \text{ if } i \ne 0\\
-w_1+1 & \text{ if } i = 0
\end{cases}.
\]
It then follows from~\cite[Prop.~1.8.9]{lazarsfeld} that
\[
\reg(\cE_i\otimes \cF)\leq \begin{cases}
(-w_0+1) + w_0 = 1 & \text{ if } i\ne 0\\
(-w_1+1)+w_0 < 1 & \text{ if } i =0
\end{cases}.
\]  
In particular, $\cE_i\otimes \cF$ has no higher cohomology for all $i$, and thus the only nonzero terms in the $\rE^1$-page of the spectral sequence for $\Phi^W_1$ appear in the $\rH^0$-row.  The $\rE_1$-page of that spectral sequence is thus a locally free resolution of $\Phi^W_1(\cF)$ which has the form
\[
\cO(-w_0)\otimes \rH^0(\PP^n, \cE_0\otimes \cF) \gets \dots \gets \cO(-w_n)\otimes \rH^0(\PP^n, \cE_n\otimes \cF)\gets 0,
\]
and hence is the desired pure resolution.
\end{proof}
Note that Example~\ref{ex:023 line} gives an example of Corollary~\ref{cor:pure resolutions}.

\section{Applications to Boij--S\"oderberg Decompositions}

In this section, we prove our main application to Boij--S\"oderberg decompositions.  

\begin{proof}[Proof of Theorem~\ref{thm:categorified special case}]
The spectral sequence $\widehat{\rE}^1_{p,q}$ for $\Phi^W_2(\cF)$ has terms $\rH^i(\PP^n, \cF(-w_i))\otimes \cE_i$ on the main diagonal, and all other terms are $0$. Next, note that in the Boij--S\"oderberg decomposition
\[
\gamma(\cF)= \sum_{i=0}^n a_i \gamma(\cE_i),
\]
the only contribution to the cohomology table in degree $-w_i$ comes from $\cE_i$ which has $\rh^i(\PP^n, \cE_i(-w_i))=N_W$ by Theorem~\ref{thm:main comparison}(2).  It follows that
\[
\rh^i(\PP^n, \cF(-w_i)) = a_i \cdot N_W.
\]

Hence, the spectral sequence $\widehat{\rE}^1_{p,q}$ has terms $\cE_i^{\oplus N_W\cdot a_i}$ along the main diagonal, and all other terms are $0$.  
It follows immediately that $\Phi^W_2(\cF)$ has a filtration of the desired form. 

For the claim about the cohomology table of $\Phi^W_2(\cF)$, we note that the given filtration of $\Phi^W_2(\cF)$ yields an upper bound of $\sum_{i=0}^n N_Wa_i\gamma(\cE_i)$ on the cohomology table of $\Phi_2^W(\cF)$; this is because cohomology tables are subadditive with respect to short exact sequences.  The lower bound is given by Proposition~\ref{prop:semicont}.
\end{proof}

\begin{example} \label{eg:BS-explain}
We now return to Example~\ref{eg:BS-cat}. Here $W=\{\cO,\cO(-2),\cO(-3)\}$ and we have
\[
\gamma(\cF) = \gamma(\cE_2) + \tfrac{1}{3}\gamma(\cE_1),
\]
where $\cE_2 = \cQ^*$ and $\cE_1 = (\Sym^2\cQ^*)(1)$.  Theorem~\ref{thm:categorified special case} proves that $\Phi^W_2(\cF)$ is an extension
\[
0\to \cE_1 \to \Phi^W_2(\cF) \to \cE_2^{\oplus 3} \to 0.
\]
However, the extension splits as claimed:
\begin{align*}
\Ext^1(\cE_2,\cE_1) = \rH^1(\cE_2^*\otimes \cE_1)=\rH^1((\bS_{2,1} \cQ^*)(2) \oplus (\Sym^3 \cQ^*)(2) )=0.
\end{align*}
In the second step we used that $\cQ \cong \cQ^*(1)$ and Pieri's rule for tensoring Schur functors, and in the third step we used Borel--Weil--Bott (Theorem~\ref{thm:bott}).
\end{example}

\begin{remark}
Note that while Corollary~\ref{cor:supernatural becomes equivariant} is a special case of Theorem~\ref{thm:categorified special case}, its proof does not require the use of Proposition~\ref{prop:semicont}.
\end{remark}

\begin{remark}\label{rmk:general}
It is natural to ask whether these techniques can be used to categorify the Boij--S\"oderberg decomposition of an arbitrary vector bundle.  For instance, one might hope for an iterative procedure in which a new $W$ can be introduced for each step of the Boij--S\"oderberg decomposition.  There are two challenges to this approach in general:  finding the correct hypotheses for when the cohomology table of $\Phi^W_2(\cF)$ is a scalar multiple of the cohomology table of $\cF$, and finding a map from $\Phi^W_2(\cF)$ to the appropriate supernatural bundle that is surjective on cohomology.  Though these both appear to be nontrivial problems, there is also room for optimism, as we know of no fundamental obstacles to such an approach working in general.
\end{remark}


\begin{thebibliography}{BEKS2}

\bibitem[AO]{ancona-ottaviani} Vincenzo Ancona, Giorgio Ottaviani, An introduction to derived categories and a theorem of Beilinson, {\it Atti Accad. Peloritana Pericolanti Cl. Sci. Fis. Mat. Natur.} {\bf 67} (1989), 99--110.

\bibitem[Be]{beilinson} A. A. Be{\u \i}linson, Coherent sheaves on ${\bf P}^n$ and problems in linear algebra, {\it Funktsional. Anal. i Prilozhen.} {\bf 12} (1978), no. 3, 68--69.

\bibitem[BEKS1]{beks-poset} Christine Berkesch, Daniel Erman, Manoj Kummini, Steven V Sam, Poset structures in Boij--S\"oderberg theory, {\it Int. Math. Res. Not. IMRN} {\bf 22} (2012), 5132--5160, \arxiv{1010.2663v2}. 

\bibitem[BEKS2]{tensor-complexes} Christine Berkesch Zamaere, Daniel Erman, Manoj Kummini, Steven~V Sam, Tensor complexes: multilinear free resolutions constructed from higher tensors, {\it J. Eur. Math. Soc. (JEMS)} {\bf 15} (2013), no.~6, 2257--2295, \arxiv{1101.4604v5}.

\bibitem[BGG]{BGG} I. N. Bernstein, I. M. Gelfand, S. I. Gelfand, Algebraic vector bundles on ${\bf P}^n$ and problems of linear algebra, {\it Funktsional. Anal. i Prilozhen.} {\bf 12} (1978), no. 3, 66--67.

\bibitem[BS1]{boij-sod-2} Mats Boij, Jonas S\"oderberg, Graded Betti numbers of Cohen--Macaulay modules and the multiplicity conjecture, {\it J. Lond. Math. Soc. (2)} {\bf 78} (2008), no. 1, 85--106, \arxiv{math/0611081v2}. 

\bibitem[BS2]{boij-sod-1} Mats Boij, Jonas S\"oderberg, Betti numbers of graded modules and the multiplicity conjecture in the non-Cohen--Macaulay case, {\it Algebra Number Theory} {\bf 6} (2012), no. 3, 437--454, \arxiv{0803.1645v1}.

\bibitem[Ca]{canonaco} Alberto Canonaco, Exceptional sequences and derived autoequivalences, \arxiv{0801.0173v1}.

\bibitem[Eis]{eisenbud} David Eisenbud, {\it Commutative Algebra with a View Toward Algebraic Geometry}, Graduate Texts in Math. {\bf 150}, Springer-Verlag, 1995.

\bibitem[EE]{eisenbud-erman-categorified} David Eisenbud, Daniel Erman, Categorified duality in Boij--S\"oderberg theory and invariants of free complexes, {\it J. Eur. Math. Soc. (JEMS)}, to appear, \arxiv{1205.0449v2}.

\bibitem[EES]{ees-filtering} David Eisenbud, Daniel Erman, Frank-Olaf Schreyer, Filtering free resolutions, {\it Compos. Math.} 149 (2013), no.~5, 754--772; corrigendum: {\it Compos. Math.} {\bf 150} (2014), no.~9, 1482--1484, \arxiv{1001.0585v3}.

\bibitem[EFS]{EFS} David Eisenbud, Gunnar Fl{\o}ystad, Frank-Olaf Schreyer, Sheaf cohomology and free resolutions over exterior algebras, {\it Trans. Amer. Math. Soc.} {\bf 355} (2003), no.~11, 4397--4426, \arxiv{math/0104203v2}.

\bibitem[EFW]{EFW} David Eisenbud, Gunnar Fl{\o}ystad, Jerzy Weyman, The existence of equivariant pure free resolutions, {\it Ann. Inst. Fourier (Grenoble)} {\bf 61} (2011), no.~3, 905--926, \arxiv{0709.1529v5}.

\bibitem[ES1]{ES2008} David Eisenbud, Frank-Olaf Schreyer, Betti numbers of graded modules and cohomology of vector bundles, {\it J. Amer. Math. Soc.} {\bf 22} (2009), no.~3, 859--888, \arxiv{0712.1843v3}.

\bibitem[ES2]{ES2010} David Eisenbud, Frank-Olaf Schreyer, Cohomology of coherent sheaves and series of supernatural bundles, {\it J. Eur. Math. Soc. (JEMS)} {\bf 12} (2010), no.~3, 703--722, \arxiv{0902.1594v1}.

\bibitem[ES3]{schreyer-eisenbud-icm} David Eisenbud, Frank-Olaf Schreyer, Betti numbers of syzygies and cohomology of coherent sheaves, {\it Proceedings of the International Congress of Mathematicians}. Volume II, 586--602, Hindustan Book Agency, New Delhi, 2010, \arxiv{1102.3559v1}.

\bibitem[ES4]{ES-river} David Eisenbud, Frank-Olaf Schreyer, The banks of the cohomology river, {\it Kyoto J. Math.} {\bf 53} (2013), no.~1, 131--144, \arxiv{1109.4591v1}.

\bibitem[Fl]{floystad} Gunnar Fl{\o}ystad, Boij-Söderberg theory: introduction and survey, {\it Progress in commutative algebra 1}, 1--54, de Gruyter, Berlin, 2012, \arxiv{1106.0381v2}.

\bibitem[FH]{fulton-harris} William Fulton, Joe Harris, {\it Representation theory. A first course}, Graduate Texts in Math. {\bf 129}, Springer-Verlag, New York, 1991.

\bibitem[GR]{GR-mutations} A. L. Gorodentsev, A. N. Rudakov, Exceptional vector bundles on projective spaces, {\it Duke Math. J.} {\bf 54} (1987), no. 1, 115--130.

\bibitem[HH]{hartshorne-hirschowitz} Robin Hartshorne, Andr\'e Hirschowitz, Cohomology of a general instanton bundle, {\it Ann. Sci. \'Ecole Norm. Sup. (4)} {\bf 15} (1982), no. 2, 365--390.

\bibitem[Hu]{huybrechts} Daniel Huybrechts, {\it Fourier-Mukai Transforms in Algebraic Geometry}, Oxford Mathematical Monographs. The Clarendon Press, Oxford University Press, Oxford, 2006.

\bibitem[KS]{kummini-sam} Manoj Kummini, Steven~V Sam, The cone of Betti tables over a rational normal curve, {\it Commutative Algebra and Noncommutative Algebraic Geometry}, 251--264, {\it Math. Sci. Res. Inst. Publ.} {\bf 68}, Cambridge Univ. Press, Cambridge, 2015, \arxiv{1301.7005v2}.

\bibitem[Laz]{lazarsfeld} Robert Lazarsfeld, {\it Positivity in algebraic geometry. I. Classical setting: line bundles and linear series}, A Series of Modern Surveys in Mathematics {\bf 48}, Springer-Verlag, Berlin, 2004.

\bibitem[SS]{symc1} Steven~V Sam, Andrew Snowden, GL-equivariant modules over polynomial rings in infinitely many variables, {\it Trans. Amer. Math. Soc.} {\bf 368} (2016), 1097--1158, \arxiv{1206.2233v3}.

\bibitem[SW]{sam-weyman} Steven~V Sam, Jerzy Weyman, Pieri resolutions for classical groups, {\it J. Algebra} {\bf 329} (2011), Special issue celebrating the 60th birthday of Corrado De Concini, 222--259, \arxiv{0907.4505v5}.

\bibitem[Wey]{weyman} Jerzy Weyman, {\it Cohomology of Vector Bundles and Syzygies}, Cambridge Tracts in Mathematics {\bf 149}, Cambridge University Press, 2003.

\end{thebibliography}
\end{document}